\documentclass[psamsfonts]{amsart}
\usepackage{amssymb,amsfonts}
\usepackage{latexsym}
\usepackage{enumerate}
\usepackage{mathrsfs}
\usepackage{url}

\newtheorem{thm}{Theorem}[section]
\newtheorem{cor}[thm]{Corollary}
\newtheorem{prop}[thm]{Proposition}

\newtheorem{conj}[thm]{Conjecture}

\theoremstyle{definition}
\newtheorem{defn}[thm]{Definition}

\theoremstyle{remark}
\newtheorem{rem}[thm]{Remark}

\numberwithin{equation}{section}

\title{Integer Complexity: The Integer Defect}
\author{Harry Altman}
\date{August 1, 2018}

\begin{document}

\newcommand{\cpx}[1]{\|#1\|}
\newcommand{\dft}{\delta}
\newcommand{\st}{{st}}
\newcommand{\xpdd}[1]{\hat{#1}}
\newcommand{\drop}{\Delta}
\newcommand{\badfac}{\kappa}
\newcommand{\cR}{R}
\newcommand{\acl}{\ell}
\newcommand{\fak}[1]{t_{#1}}

\newcommand{\NUM}{{48}}
\newcommand{\TWONUM}{{96}}

\newcommand{\N}{{\mathbb N}}
\newcommand{\R}{{\mathbb R}}
\newcommand{\Z}{{\mathbb Z}}
\newcommand{\Q}{{\mathbb Q}}
\newcommand{\sS}{{\mathcal S}}
\newcommand{\sT}{{\mathcal T}}

\newcommand{\floor}[1]{{\left\lfloor #1 \right\rfloor}}
\newcommand{\ceil}[1]{{\left\lceil #1 \right\ceil}}

\begin{abstract}
Define $\cpx{n}$ to be the \emph{complexity} of $n$, the smallest number of ones
needed to write $n$ using an arbitrary combination of addition and
multiplication.  John Selfridge showed that $\cpx{n}\ge 3\log_3 n$ for all $n$,
leading this author and Zelinsky to define the \emph{defect} of $n$, $\dft(n)$,
to be the difference $\cpx{n}-3\log_3 n$.  Meanwhile, in the study of addition
chains, it is common to consider $s(n)$, the number of small steps of $n$,
defined as $\ell(n)-\lfloor\log_2 n\rfloor$, an integer quantity.  So here we
analogously define $D(n)$, the \emph{integer defect} of $n$, an integer version
of $\dft(n)$ analogous to $s(n)$.  Note that $D(n)$ is not the same as $\lceil
\dft(n) \rceil$.

We show that $D(n)$ has additional meaning in terms of the defect well-ordering
considered in \cite{paperwo}, in that $D(n)$ indicates which powers of $\omega$
the quantity $\dft(n)$ lies between when one restricts to $n$ with $\cpx{n}$
lying in a specified congruence class modulo $3$.  We also determine all numbers
$n$ with $D(n)\le 1$, and use this to generalize a result of Rawsthorne
\cite{Raws}.
\end{abstract}

\maketitle

\section{Introduction}
\label{intro}

The \emph{complexity} of a natural number $n$, denoted $\cpx{n}$, is the least
number of $1$'s needed to write it using any combination of addition and
multiplication, with the order of the operations specified using  parentheses
grouped in any legal nesting.  For instance, $n=11$ has a complexity of $8$,
since it can be written using $8$ ones as \[ 11=(1+1+1)(1+1+1)+1+1,\] but not
with any fewer than $8$.  This notion was implicitly introduced in 1953 by Kurt
Mahler and Jan Popken \cite{MP}, and was later popularized by Richard Guy
\cite{Guy, UPINT}.

Integer complexity is approximately logarithmic; it satisfies the bounds
\begin{equation*}\label{eq1}
3 \log_3 n= \frac{3}{\log 3} \log  n\le \cpx{n} \le \frac{3}{\log 2} \log n  ,\qquad n>1.
\end{equation*}
The lower bound can be deduced from the results of Mahler and Popken, and was
explicitly proved by John Selfridge \cite{Guy}. It is attained with equality for
$n=3^k$ for all $k \ge1$.  The upper bound can be obtained by writing $n$ in
binary and finding a representation using Horner's algorithm. It is not sharp,
and the constant $\frac{3}{\log2} $ can be improved for large $n$ \cite{upbds}.

Based on the lower bound, this author and Zelinsky \cite{paper1} introduced the
notion of the \emph{defect} of $n$, denoted $\dft(n)$, which is the difference
$\cpx{n}-3\log_3 n$.  Subsequent work \cite{paperwo} showed that the set of
defects is in fact a well-ordered subset of the real line, with order type
$\omega^\omega$.

However, it is worth considering the result of Selfridge in more detail:
\begin{thm}[Selfridge]
\label{selfridge}
For any $k\ge1$, let $E(k)$ be the largest number that can be made with $k$
ones, i.e., the largest $n$ with $\cpx{n}\le k$.  Then:
\begin{enumerate}
\item If $k=1$, then $E(k)=1$.
\item If $k\equiv 0\pmod{3}$, then $E(k)=3^{k/3}$.
\item If $k\equiv 1\pmod{3}$ and $k>1$, then $E(k)=4\cdot3^{(k-4)/3}$.
\item If $k\equiv 2\pmod{3}$, then $E(k)=2\cdot3^{(k-2)/3}$.
\end{enumerate}
\end{thm}

(This result is also a special case of the results of Mahler and Popken
\cite{MP}.)  From this one can of course derive the lower bound $\cpx{n}\ge
3\log_3 n$, but what if one wanted an integer version?  We make the following
definition:

\begin{defn}
Given a natural number $n$, we define $L(n)$ to be the largest $k$ such that
$E(k)\le n$.
\end{defn}

With this, we define:

\begin{defn}
For a natural number $n$, we define the \emph{integer defect} of $n$, denoted
$D(n)$, to be the difference $\cpx{n}-L(n)$.
\end{defn}

Because of Theorem~\ref{selfridge}, $L(n)$ is quite easy to compute (see
Proposition~\ref{computL}), and hence
if one knows $\cpx{n}$ then $D(n)$ is also easy to compute.  Note that while
we consider $D(n)$ to be an integer analogue of $\dft(n)$, it is not in general
equal to
$\lceil \dft(n) \rceil$; see Theorem~\ref{dtoD} for the precise relation.
However it's not immediately obvious that $D(n)$ has any actual significance.
In fact, however, the integer defect of a number tells you about its position in
the well-ordering of defects.

\begin{rem}
$L(k)$ is not the best lower bound we can get from
Theoerem~\ref{selfridge}; that would instead be the smallest $k$ such that
$E(k)\ge n$, which we might denote $L'(n)$.  ($L'(n)$ could also be defined as
the minimum of $\cpx{m}$ over all $m\ge n$.)  For reasons that will become clear
later, though, we will prefer to discuss $L$ rather than $L'$.  In any case,
$L'(n)=L(n)+1$ unless $n=E(k)$ for some $k$, in which case $L'(n)=L(n)=k$, so
one can easily convert any results expressed in the one formulation to the
other. One could consider a similar $D'(n)$ as well, but we will not do
that either.
\end{rem}

\subsection{The sets $\mathscr{D}^0$, $\mathscr{D}^1$, and $\mathscr{D}^2$ and
the main result}
\label{thm1sec}

As has been noted above, if we define $\mathscr{D}$ to be the set of all
defects, then as a subset of the real line this set is well-ordered and has
order type $\omega^\omega$.  However, more specific theorems are proved in
\cite{paperwo}.  We will need the following definition:

\begin{defn}
\label{dadef}
If $a$ is a congruence class modulo $3$, we define
\[ \mathscr{D}^a = \{ \dft(n) : \cpx{n}\equiv a\pmod{3},~~n\ne 1\}. \]
\end{defn}

\begin{rem}
The number $n=1$ is excluded from $\mathscr{D}^1$ because it is dissimilar to
other numbers whose complexity is congruent to $1$ modulo $3$.  Unlike other
numbers which are $1$ modulo $3$, the number $1$ cannot be written as $3j+4$
for some $j\ge0$, and so the largest number that can be made with a single $1$
is simply $1$, rather than $4\cdot 3^j$.
\end{rem}

In fact the sets $\mathscr{D}^a$ for $a=0,1,2$ are disjoint, and so together
with $\{1\}$ form a partition of $\mathscr{D}$.

Moreover in \cite{paperwo} it was proved:
\begin{thm}
\label{wothm}
For $a=0,1,2$, the sets $\mathscr{D}^a$ are all well-ordered, each with order
type $\omega^\omega$.
\end{thm}

It is these sets, the $\mathscr{D}^a$,
that $D(n)$ will tell us about the position of $\dft(n)$ in.  We show:

\begin{thm}[Main theorem]
\label{mainthm}
Let $n>1$ be a natural number
Let $\zeta$ be the order type of $\mathscr{D}^{\cpx{n}}\cap[0,\dft(n))$.  Then
$D(n)$ is equal to the smallest $k$ such that $\zeta<\omega^k$.
\end{thm}

As mentioned above, $D(n)$ is easy to compute, so this theorem gives an way to
easily compute where around $\dft(n)$ falls in the ordering on $\mathscr{D}^a$.

We will also prove a version of this theorem for the stable integer defect; see
Sections~\ref{secstab} and \ref{secdefn}.

It's worth comparing this theorem to what was already known.  It was proved in
\cite{paperwo} that
the limit of the initial $\omega^k$ elements of $\mathscr{D}$ is equal to $k$.
This raises the question -- just what is the
limit of the initial $\omega^k$ elements of $\mathscr{D}^a$?  It was further
shown in \cite{paperwo} that when $k\equiv a\pmod{3}$ this limit is equal to
$k$, but what about otherwise?

In this paper we will answer this question:
\begin{thm}
\label{chgoverpt}
The limit of the initial $\omega^k$ elements of $\mathscr{D}^a$ is equal to $k$
if $k-a\equiv 0\pmod{3}$; it is equal to $k+\dft(2)$ if $k-a\equiv 1\pmod{3}$;
and it is equal to $k+2\dft(2)$ if $k-a\equiv 2\pmod{3}$.
\end{thm}

In fact, Theorem~\ref{chgoverpt} will be used to prove
Theorem~\ref{mainthm}.  See Section~\ref{mainsec} for more general
statements.  Further generalizations will appear in a future paper \cite{stab}.

\subsection{Generalizing Rawsthorne's theorem}
\label{rawsintro}

We know how to compute $E(k)$, the highest number of complexity at most $k$ (or
exactly $k$), but what about the next highest?  This question was answered by
Daniel Rawsthorne \cite{Raws} in 1989:

\begin{thm}[Rawsthorne]
\label{rawsthm}
For any $k\ge 8$, the highest number of complexity at most $k$ other than
$E(k)$ itself is $\frac{8}{9}E(k)$, and this number has complexity exactly $k$.
\end{thm}

In this paper we generalize this result.  First, a definition:

\begin{defn}
Given $r\ge 0$ and $k\ge 1$, we define $E_r(k)$ to be the $r$'th largest number
of complexity at most $k$.  We will $0$-index here, so that by definition
$E_0(k)=E(k)$, and Theorem~\ref{rawsthm} gives a formula for $E_1(k)$.
\end{defn}

Then, with this, we show:

\begin{thm}
\label{tablethm}
Given $r\ge 0$, and $a$ a congruence class modulo $3$, there exists $K_{r,a}>1$
and $h_{r,a}\in\mathbb{Q}$ such that for $k\ge K_{r,a}$ with $k\equiv a
\pmod{3}$, we have $E_r(k)=h_{r,a} E(k)$, and these $h_{r,a}$ and $K_{r,a}$ are
as given by Tables~\ref{table0}, \ref{table2}, and \ref{table1}. Moreover, for
such $r$ and $k$, we have $E_r(k)>E(k-1)$ and therefore $\cpx{E_r(k)}=k$ (and
thus for such $r$ and $k$, $E_r(k)$ is not just the $r$'th largest number with
complexity at most $k$, but the $r$'th largest number with complexity exactly
$k$).
\end{thm}

\begin{table}[htb]
\caption{Table of $h_r$ and $K_r$ for $k\equiv 0\pmod{3}$.}
\label{table0}
\begin{tabular}{|r|r|r|}
\hline
$r$ & 					$h_{r,0}$ & 	$K_{r,0}$ \\
\hline
$0$ &					$1$ & 		$3$ \\
$1$ &					$8/9$ &		$6$ \\
$2$ &					$64/81$ &	$12$ \\
$3$ &					$7/9$ &		$12$ \\
$4$ &					$20/27$ &	$12$ \\
$5$ &					$19/27$ &	$12$ \\
$6$ &					$512/729$ &	$18$ \\
$7$ &					$56/81$ &	$18$ \\
$8$ &					$55/81$ &	$18$ \\
$9$ &					$164/243$ &	$18$ \\
$10$ &					$163/243$ &	$18$ \\
$(\mathrm{for}\: n\ge6) \quad 2n-1$ &	$2/3+2/3^n$ &	$3n$ \\
$(\mathrm{for}\: n\ge6) \quad 2n$ &	$2/3+1/3^n$ &	$3n$ \\
\hline
\end{tabular}
\end{table}

\begin{table}[htb]
\caption{Table of $h_r$ and $K_r$ for $k\equiv 2\pmod{3}$.}
\label{table2}
\begin{tabular}{|r|r|r|}
\hline
$r$ &					$h_{r,2}$ &		$K_{r,2}$ \\
\hline
$0$ &					$1$ &			$2$ \\
$1$ &					$8/9$ &			$8$ \\
$2$ &					$5/6$ &			$8$ \\
$3$ &					$64/81$ &		$14$ \\
$4$ &					$7/9$ &			$14$ \\
$5$ &					$20/27$ &		$14$ \\
$6$ &					$13/18$ &		$14$ \\
$7$ &					$19/27$ &		$14$ \\
$8$ &					$512/729$ &		$20$ \\
$9$ &					$56/81$ &		$20$ \\
$10$ &					$37/54$ &		$20$ \\
$11$ &					$55/81$ &		$20$ \\
$12$ &					$164/243$ &		$20$ \\
$13$ &					$109/162$ &		$20$ \\
$14$ &					$163/243$ &		$20$ \\
$(\mathrm{for}\: n\ge 6)\quad3n-3$ &	$2/3+2/3^n$ &		$3n+2$ \\
$(\mathrm{for}\: n\ge 6)\quad3n-2$ &	$2/3+1/(2\cdot3^{n-1})$&$3n+2$ \\
$(\mathrm{for}\: n\ge 6)\quad3n-1$ &	$2/3+1/3^n$ &		$3n+2$ \\
\hline
\end{tabular}
\end{table}

\begin{table}[htb]
\caption{Table of $h_r$ and $K_r$ for $k\equiv 1\pmod{3}$ with $k>1$.}
\label{table1}
\begin{tabular}{|r|r|r|}
\hline
$r$ &					$h_{r,1}$ &		$K_{r,1}$ \\
\hline
$0$ &					$1$ &			$4$ \\
$1$ &					$8/9$ &			$10$ \\
$2$ &					$5/6$ &			$10$ \\
$3$ &					$64/81$ &		$16$ \\
$4$ &					$7/9$ &			$16$ \\
$5$ &					$41/54$ &		$16$ \\
$(\mathrm{for}\: n\ge 4)\quad n+2$ &	$3/4+1/(4\cdot3^n)$ &	$3n+4$ \\
\hline
\end{tabular}
\end{table}

Note that Tables~\ref{table0}, \ref{table2}, and \ref{table1} don't list the
regular pattern in the $h_{r,a}$ until such point as $K_{r,a}$ also becomes
regular; for tables based solely on $h_{r,a}$, see Tables~\ref{table0r},
\ref{table2r}, and \ref{table1r}.

What does Theorem~\ref{tablethm} have to do with integer defect?  Well, the
numbers $h_{r,a}E(k)$ appearing in this theorem are almost exactly the numbers
$n$ with $D(n)\le1$; see Proposition~\ref{coincicor} for a precise statement.

After all, by Theorem~\ref{mainthm}, the numbers $n$ with $D(n)\le 1$ are
precisely those $n$ whose $\dft(n)$ lie in the initial $\omega$ of
$\mathscr{D}^{\cpx{n}}$.  So if one fixes a particular $k$, then going down the
set of $n$ with $\cpx{n}=k$ corresponds to going up the set of defects $\dft(n)$
of $n$ with $\cpx{n}=k$; and assuming $k$ is large enough relative to how far up
or down you want to go, this is just looking at $\mathscr{D}^k$.  And if we
count up one at a time, then -- again, assuming $k$ is sufficiently large
relative to how far out we count -- we will stay within the initial $\omega$ of
$\mathscr{D}^k$.  So with a classification of numbers $n$ such that $D(n)\le 1$,
one can determine the $E_r(k)$.  (Indeed, one can also do the reverse.)

Note that Theorem~\ref{rawsthm} also works for $k=6$, so if one wants to break
it down by the residue of $k$ modulo $3$, one could say it works for $k\ge 6$
with $k\equiv 0\pmod{3}$, for $k\ge 8$ with $k\equiv 2\pmod{3}$, and for $k\ge
10$ with $k\equiv 1\pmod{3}$.  (Indeed, this is what we have done in
Tables~\ref{table0}, \ref{table2}, and \ref{table1}.)  Note how all three of
these correspond to $k$ exactly large enough for $E(k)$ to be divisible by $9$,
as per the last part of Theorem~\ref{tablethm}.

One thing worth noting here is that the formulae for $E_0(k)$ and $E_1(k)$, as
originally proven by Selfridge and Rawsthorne respectively, were both originally
proven directly by induction on $k$.  Whereas here we have proven
Theorem~\ref{tablethm} by a different method, namely, analysis of defects.
(Although this analysis of defects in turn depends on Rawsthorne's formula for
$E_1(k)$ to serve as a base case; see \cite{paper1}.)  This raises the question
of whether a similar inductive proof for general $E_r(k)$ could be done now that
the formulae for them are known.  (In fact this author originally proved these
formulae by a different method entirely, that of analyzing certain
transformations of expression, so other methods certainly are possible.)

\subsection{Low-defect polynomials and numbers of low defect}
\label{intropoly}

In order to prove Theorem~\ref{mainthm}, we make use of the idea of
low-defect polynomials from \cite{paperwo,theory}.  A low-defect polynomial is a
particular type of multilinear polynomial; see Section~\ref{polysec} for
details.  In \cite{paperwo} it is proved that, given any positive real number
$s$, one can write down a finite set of low-defect polynomials $\sS$
such that every number $n$ with $\dft(n)<s$ can be written in the form
$f(3^{n_1},\ldots,3^{n_d})3^{n_{d+1}}$ for some $f\in\sS$; and that,
moreover, such an $n$ can always be represented ``efficiently'' in such a
fashion.  Moreover, one can choose $\sS$ such that for any
$f\in\sS$, one has $\deg f\le s$.  (Note that the degree of a low-defect
polynomial is always equal to the number of variables it is in, since low-defect
polynomials are multilinear and always include a term containing all the
variables.)

Using this fact about low-defect polynomials, this author proved in \cite{paperwo} that the set $\mathscr{D}$ is
well-ordered with order type $\omega^\omega$, as well as the more specific
Theorem~\ref{wothm} mentioned above, and other results mentioned above such as
that the limit of the initial $\omega^k$ defects is equal to $k$.  However, this
is not enough to prove the more specific theorems shown in this paper, such as
Theorem~\ref{chgoverpt}.  But in \cite{theory} an improvement was shown, that we
can in fact take $\sS$ such that for all $f\in\mathscr{T}$, one has
$\dft(f)\le s$; here $\dft(f)$ is a number that bounds above $\dft(n)$ for any
$n$ represented by $f$ in the fashion described above; again, see \ref{polysec}
for more on this.

On top of that, it was shown in \cite{theory} that $\dft(f)\ge \deg f +
\dft(m)$, where $m$ is the leading coefficient of $f$.  Putting this together,
one gets the inequality \[ \deg f + \dft(m) \le s.\] It's this stronger
inequality that allows us to prove Theorem~\ref{mainthm}, where the
inequality $\deg f\le s$ would not be enough.  To see why this inequality is so
helpful, say we're given $s$ and we pick $\sS$ as described above.  Then
if $f\in\sS$, one of two things must be true: Either $\deg f<\lfloor
s\rfloor$, in which case $f$ does not make much of a contribution to
$\mathscr{D}\cap [0,s)$ compared to polynomials of higher degree; or $\deg
f=\lfloor s\rfloor$, in which case $\dft(m)$ is at most the fractional part of
$r$, a number which is less than $1$.  Since there are only finitely many
defects below any given number less than $1$, this puts substantial constraints
on $m$ and therefore on $f$, in ways that the weaker inequality $\deg f\le s$
does not.  This allows us to prove Theorem~\ref{chgoverpt}.

Note that the method we use to turn the results of \cite{theory} into
Theorem~\ref{mainthm} actually has much more power than we use in this
paper; but an exploration of the full power of this method would take us too far
away from the subject of $D(n)$, and so will be detailed in a future paper
\cite{stab}.

\subsection{A quick note on stabilization}
\label{secstab}

An important property satisfied by integer integer complexity is the phenomenon
of stabilization.  Because one has $\cpx{3^k}=3k$ for $k>1$, as well as that
$\cpx{2\cdot3^k}=2+3k$ and $\cpx{4\cdot3^k}=4+3k$, one might hope that in
general the equation $\cpx{3n}=\cpx{n}+3$ holds for all $n>1$.  Unfortunately
that is not the case; for instance, for $n=107$, one has $\cpx{107}=16$, but
$\cpx{321}=18$.  Another counterexample is $n=683$, for which one has
$\cpx{683}=22$, but $\cpx{2049}=23$.  There are even cases where
$\cpx{3n}<\cpx{n}$, such as $n=4721323$, which has $\cpx{3n}=\cpx{n}-1$.

And yet the initial hope is not entirely in vain.  In \cite{paper1}, it was
proved:

\begin{thm}
\label{basicstab}
For any natural number $n$, there exists $K\ge 0$ such that, for any $k\ge K$,
\[ \cpx{3^k n}=3(k-K)+\cpx{3^K n}. \]
\end{thm}

Based on this, we define:

\begin{defn}
A number $m$ is called \emph{stable} if $\cpx{3^k m}=3k+\cpx{m}$ holds for every
$k \ge 0$.
Otherwise it is called \emph{unstable}.
\end{defn}

So, we can restate Theorem~\ref{basicstab} by saying, for any $n$, there is some
$K$ such that $3^K n$ is stable.

This allows us to define stable or stabilized analogues of many of the concepts
and discussed above, and prove stabilized analogues of the theorems discussed in
Section~\ref{thm1sec}.  See Sections~\ref{subsecdft} and \ref{secdefn} for the
relevant definitions, and Section~\ref{mainsec} for the versions of the main
theorems generalized to cover the stabilized case as well.

\subsection{Discussion: Comparison to addition chains}

In order to make sense of Theorem~\ref{mainthm}, it is helpful to introduce
an analogy to addition chains, a different notion of complexity which is similar
in spirit but different in detail.
An \emph{addition chain} for $n$ is defined to be a sequence
$(a_0,a_1,\ldots,a_r)$ such that $a_0=1$, $a_r=n$, and, for any $1\le k\le r$,
there exist $0\le i, j<k$ such that $a_k = a_i + a_j$; the number $r$ is called
the length of the addition chain.  The shortest length among addition chains for
$n$, called the \emph{addition chain length} of $n$, is denoted $\acl(n)$.
Addition chains were introduced in 1894 by H.~Dellac \cite{Dellac} and
reintroduced in 1937 by A.~Scholz \cite{aufgaben}; extensive surveys on the
topic can be found in Knuth \cite[Section 4.6.3]{TAOCP2} and Subbarao
\cite{subreview}.

The notion of addition chain length has obvious similarities to that of integer
complexity; each is a measure of the resources required to build up the number
$n$ starting from $1$.  Both allow the use of addition, but integer complexity
supplements this by allowing the use of multiplication, while addition chain
length supplements this by allowing the reuse of any number at no additional
cost once it has been constructed.  Furthermore, both measures are approximately
logarithmic; the function $\acl(n)$ satisfies
\[ \log_2 n \le \acl(n) \le 2\log_2 n. \]
 
A difference worth noting is that $\acl(n)$ is actually known to be asymptotic
to $\log_2 n$, as was proved by Brauer \cite{Brauer}, but the function $\cpx{n}$
is not known to be asymptotic to $3\log_3 n$; the value of the quantity
$\limsup_{n\to\infty} \frac{\cpx{n}}{\log n}$ remains unknown.

Nevertheless, there are important similarities between integer complexity and
addition chains.  As mentioned above, the set of all integer complexity defects
is a well-ordered subset of the real numbers, with order type $\omega^\omega$.
We might also define the notion of \emph{addition chain defect}, defined by
\[\dft^\acl(n):=\acl(n)-\log_2 n;\]
for as shown in \cite{adcwo}, the well-ordering theorem for integer complexity
has an analogue for addition chains:

\begin{thm}[Addition chain well-ordering theorem]
Let $\mathscr{D}^\acl$ denote the set $\{ \dft^\acl(n) : n \in \mathbb{N}
\}$.  Then considered as a subset of the real numbers, $\mathscr{D}^\acl$ is
well-ordered and has order type $\omega^\omega$.
\end{thm}

More commonly, however, it is not $\dft^\acl(n)$ that has been studied, but
rather $s(n)$, the number of \emph{small steps} of $n$, which is defined to be
$\acl(n)-\lfloor\log_2\rfloor$, or equivalently $\lceil\dft^\acl(n)\rceil$.
The quantity $D(n)$ that we introduce seems to play a role in integer complexity
similar to $s(n)$ in the study of addition chains.  Now, unlike with $s(n)$ and
$\dft^\acl(n)$, $D(n)$ is not simply $\lceil\dft(n)\rceil$; for instance,
$D(56)=1$ even though $\dft(56)>1$.  (Although Theorem~\ref{dtoD} will show how
$D(n)$ is in a certain sense almost $\lceil\dft(n)\rceil$.)  But, there are
further analogies.

Analogous to Theorem~\ref{basicstab}, we have (from \cite{adcwo}) the following:

\begin{thm}
\label{adcstab}
For any natural number $n$, there exists $K\ge 0$ such that, for any $k\ge K$,
\[ \acl(2^k n)=(k-K)+\acl(2^K n). \]
\end{thm}

So we define a number $n$ to be \emph{$\acl$-stable} if for any $k$, one has
$\acl(2^k n)=k+\acl(n)$; then Theorem~\ref{adcstab} says that for any $n$, there
is some $K$ such that $2^K n$ is $\acl$-stable.  This allows us to formulate a
stabilized version of the previous analogy -- and of the ones to follow.

In \cite{adcwo}, this author conjectured:
\begin{conj}
For each whole number $k$, $\mathscr{D}^\acl\cap[0,k]$ has order type
$\omega^k$.
\end{conj}

In other words, this conjecture states that the limit of the initial $\omega^k$
addition chain defects is equal to $k$.
If true, this would mean that $s(n)$ plays the same role for
$\mathscr{D}^\acl$ as $D(n)$ does for the $\mathscr{D}^a$, that $s(n)$ is the
smallest $k$
such that the order type of $\mathscr{D}^\acl\cap[0,\dft^\acl(n))$ is less than
$\omega^k$.

One similarly based on conjectures in \cite{adcwo} gets analogies between
$D_\st(n)$ and $s_\st(n)$ and how they determine position in $\mathscr{D}^a_\st$
and $\mathscr{D}^\acl_\st$, respectively; see Section~\ref{secdefn} for
definitions of these.

It's worth noting here one important difference between these two cases: in the
integer complexity case, we need to split things into congruence classes modulo
$3$ based on $\cpx{n}$.  This has no analogue in the addition chain case.  The
difference comes from a difference in certain fundamental inequalities that
these quantities obey.  Integer complexity obeys $\cpx{3n}\le\cpx{n}+3$, with
equality if and only if $\dft(3n)=\dft(n)$.  The addition chain analogue of this
is that one has $\acl(2n)\le\acl(n)+1$, with equality if and only if
$\dft^\acl(2n)=\dft^\acl(n)$.  The result \cite{adcwo,paper1} is that if we have
two numbers $m$ and $n$ with $\dft^\acl(n)=\dft^\acl(m)$, then one must have
$m=2^k n$ for some $k\in\mathbb{Z}$; and if we have two numbers $m$ and $n$ with
$\dft(n)=\dft(m)$, then one must have $m=3^k n$ for some $k\in\mathbb{Z}$.
However in the latter case we must also have $\cpx{m}\equiv\cpx{n}\pmod{3}$;
this is why the sets $\mathscr{D}^a$ are disjoint.  In
the addition chain case there is no such congruence requirement; $\acl(n)$ and
$\acl(m)$ need only be congruent modulo $1$, which is no requirement at all, so
splitting up $\mathscr{D}^\acl$ in a similar manner does not make sense.  The
set $\mathscr{D}^\acl$ already covers the one and only congruence class that
exists in the addition chain case.

But it is not only our primary theorem but also our secondary theorem here that
has an analogues for addition chains, and in this case the analogy does not rely
on any conjectures.
While the hypothesis that
the order type of $\mathscr{D}^\acl\cap[0,k]$ is equal to $\omega^k$ remains a
conjecture, that this holds for $k\le 2$ -- and in particular that it holds for
$k=1$ -- was proven in \cite{adcwo}.  This means that just as we can look at the
first $\omega$ elements of each $\mathscr{D}^a$ in order to determine the
$r$'th-highest number of complexity $k$, we can look at the first $\omega$
elements of $\mathscr{D}^\acl$ to determine the $r$'th-highest number of
addition chain length $k$ (or at most $k$, which in these cases is the same
thing).  (Again, here $k$ must be sufficiently large relative
to $r$.  Also, again here we are using the convention that $r$ starts at $0$
rather than $1$.)

Specifically, it's an easy corollary of the classification of numbers with
$s(n)\le 1$ (due to Gioia et al.~\cite{sub1962}) that:

\begin{thm}
For $k\ge r+1$ (or for $k\ge0$ when $r=0$), the $r$'th-largest number of
addition chain length $k$ is $(\frac{1}{2}+\frac{1}{2^{r+1}})2^k$.
\end{thm}

Obviously here the fraction $\frac{1}{2}+\frac{1}{2^{r+1}}$ plays the role of
the $h_r$ and $r+1$ plays the role of $K_r$; unlike with integer complexity,
there are no irregularities here,
just a single straightforward infinite family.  (And note how the analogue of
the $K_r$ increases in what is mostly steps of $1$, rather than mostly steps of
$3$ like the actual $K_r$, because once again with addition chains there's only
one congruence class.)
For more on the analogy between integer complexity and addition chains,
particularly with regard to their sets of defects, one may see \cite{theory}.

\section{Integer complexity, well-ordering, and low-defect polynomials}
\label{polysec}

In this section we summarize the results of \cite{paperwo,theory,paper1}
that we will need later regarding the defect $\dft(n)$; the stable complexity
$\cpx{n}_\st$ and stable defect $\dft_\st(n)$ described below; and low-defect
polynomials.

\subsection{The defect and stability}
\label{subsecdft}

First, some basic facts about the defect:

\begin{thm}
\label{oldprops}
We have:
\begin{enumerate}
\item For all $n$, $\dft(n)\ge 0$.
\item For $k\ge 0$, $\dft(3^k n)\le \dft(n)$, with equality if and only if
$\cpx{3^k n}=3k+\cpx{n}$.  The difference $\dft(n)-\dft(3^k n)$ is a nonnegative
integer.
\item A number $n$ is stable if and only if for any $k\ge 0$, $\dft(3^k
n)=\dft(n)$.
\item If the difference $\dft(n)-\dft(m)$ is rational, then $n=m3^k$ for some
integer $k$ (and so $\dft(n)-\dft(m)\in\mathbb{Z}$).
\item Given any $n$, there exists $k$ such that $3^k n$ is stable.
\item For a given defect $\alpha$, the set $\{m: \dft(m)=\alpha \}$ has either
the form $\{n3^k : 0\le k\le L\}$ for some $n$ and $L$, or the form $\{n3^k :
0\le k\}$ for some $n$.  The latter occurs if and only if $\alpha$ is the
smallest defect among $\dft(3^k n)$ for $k\in \mathbb{Z}$.
\item If $\dft(n)=\dft(m)$, then $\cpx{n}=\cpx{m} \pmod{3}$.
\item $\dft(1)=1$, and for $k\ge 1$, $\dft(3^k)=0$.  No other integers occur as
$\dft(n)$ for any $n$.
\item If $\dft(n)=\dft(m)$ and $n$ is stable, then so is $m$.
\end{enumerate}
\end{thm}

\begin{proof}
Parts (1) through (8), excepting part (3), are just Theorem~2.1 from
\cite{paperwo}.  Part (3) is Proposition~12 from \cite{paper1}, and part (9) is
Proposition~3.1 from \cite{paperwo}.
\end{proof}

We will want to consider the set of all defects:

\begin{defn}
We define the \emph{defect set} $\mathscr{D}$ to be $\{\dft(n):n\in\N\}$, the
set of all defects.
\end{defn}

We also defined $\mathscr{D}^a$, for $a$ a congruence class modulo $3$, in
Definition~\ref{dadef} earlier.

The paper \cite{paperwo} also defined the notion of a \emph{stable defect}:

\begin{defn}
We define a \emph{stable defect} to be the defect of a stable number, and define
$\mathscr{D}_\st$ to be the set of all stable defects.  Also, for $a$ a
congruence class modulo $3$, we define $\mathscr{D}^a_\st=\mathscr{D}^a \cap
\mathscr{D}_\st$.
\end{defn}

Because of part (9) of Theorem~\ref{oldprops}, this definition makes sense; a
stable defect $\alpha$ is not just one that is the defect of some stable number,
but one for which any $n$ with $\dft(n)=\alpha$ is stable.  Stable defects can
also be characterized by the following proposition from \cite{paperwo}:

\begin{prop}
\label{modz1}
A defect $\alpha$ is stable if and only if it is the smallest
$\beta\in\mathscr{D}$ such that $\beta\equiv\alpha\pmod{1}$.
\end{prop}

We can also define the \emph{stable defect} of a given number, which we denote
$\dft_\st(n)$.

\begin{defn}
For a positive integer $n$, define the \emph{stable defect} of $n$, denoted
$\dft_\st(n)$, to be $\dft(3^k n)$ for any $k$ such that $3^k n$ is stable.
(This is well-defined as if $3^k n$ and $3^\ell n$ are stable, then $k\ge \ell$
implies $\dft(3^k n)=\dft(3^\ell n)$, and $\ell\ge k$ implies this as well.)
\end{defn}

Note that the statement ``$\alpha$ is a stable defect'', which earlier we were
thinking of as ``$\alpha=\dft(n)$ for some stable $n$'', can also be read as the
equivalent statement ``$\alpha=\dft_\st(n)$ for some $n$''.

Similarly we have the stable complexity:
\begin{defn}
For a positive integer $n$, define the \emph{stable complexity} of $n$, denoted
$\cpx{n}_\st$, to be $\cpx{3^k n}-3k$ for any $k$ such that $3^k n$ is stable.
\end{defn}

We then have the following facts relating the notions of $\cpx{n}$, $\dft(n)$,
$\cpx{n}_\st$, and $\dft_\st(n)$:

\begin{prop}
\label{stoldprops}
We have:
\begin{enumerate}
\item $\dft_\st(n)= \min_{k\ge 0} \dft(3^k n)$
\item $\dft_\st(n)$ is the smallest $\alpha\in\mathscr{D}$ such that
$\alpha\equiv \dft(n) \pmod{1}$.
\item $\cpx{n}_\st = \min_{k\ge 0} (\cpx{3^k n}-3k)$
\item $\dft_\st(n)=\cpx{n}_\st-3\log_3 n$
\item $\dft_\st(n) \le \dft(n)$, with equality if and only if $n$ is stable.
\item $\cpx{n}_\st \le \cpx{n}$, with equality if and only if $n$ is stable.
\item $\cpx{3n}_\st = \cpx{n}_\st+3$
\item If $\dft_\st(n)=\dft_\st(m)$, then $\cpx{n}_\st\equiv\cpx{m}_\st\pmod{3}$.
\end{enumerate}
\end{prop}

\begin{proof}
Statements (1)-(6) are just Propositions~3.5, 3.7, and 3.8 from \cite{paperwo}.
Statement (7) follows from the definition of stable complexity; if $3^k n$ is
stable, then $\cpx{3n}_\st=\cpx{3^k n}-3(k-1)=\cpx{3^k n}-3k+3=\cpx{n}_\st+3$.
To prove statement (8), note that if $\dft_\st(n)=\dft_\st(m)$, then by
statement (2) one has $\dft(n)\equiv\dft(m)\pmod{1}$, and so by
Propostion~\ref{oldprops}, one has that $n=m3^k$ for some $k\in\mathbb{Z}$, and
so $\cpx{n}_\st=\cpx{m}_\st+3k$.
\end{proof}

Note, by the way, that just as $\mathscr{D}_\st$ can be characterized either as
defects $\dft(n)$ with $n$ stable or as defects $\dft_\st(n)$ for any $n$,
$\mathscr{D}^a_\st$ can be characterized either as defects $\dft(n)$ with $n$
stable and $\cpx{n}\equiv a\pmod{3}$, or as defects $\dft_\st(n)$ for any $n$
with $\cpx{n}_\st\equiv a\pmod{3}$.

Three defects that will be particularly important in this paper are the smallest
three defects:
\begin{prop}
\label{small3}
\[ \mathscr{D}\cap [0,2\dft(2)] = \{0,\dft(2),2\dft(2)\}. \]
\end{prop}

\begin{proof}
Proposition~37 from \cite{paper1} tells us that the only leaders with defect
less than $3\dft(2)$ are $3$, $2$, and $4$, which respectively have defects $0$,
$\dft(2)$, and $2\dft(2)$.
\end{proof}

\subsection{Low-defect polynomials}
\label{polysubsec}

As has been mentioned in Section~\ref{intropoly}, we are going to represent the
set of numbers with defect at most $r$ by substituting in powers of $3$ into
certain multilinear polynomials we call \emph{low-defect polynomials}.  We will
associate with each one a ``base complexity'' to form a \emph{low-defect pair}.
In this section we will review the basic properties of these polynomials.
First, their definition:

\begin{defn}
\label{polydef}
We define the set $\mathscr{P}$ of \emph{low-defect pairs} as the smallest
subset of $\Z[x_1,x_2,\ldots]\times \N$ such that:
\begin{enumerate}
\item For any constant polynomial $k\in \N\subseteq\Z[x_1, x_2, \ldots]$ and any
$C\ge \cpx{k}$, we have $(k,C)\in \mathscr{P}$.
\item Given $(f_1,C_1)$ and $(f_2,C_2)$ in $\mathscr{P}$, we have $(f_1\otimes
f_2,C_1+C_2)\in\mathscr{P}$, where, if $f_1$ is in $d_1$ variables and $f_2$ is
in $d_2$ variables,
\[ (f_1\otimes f_2)(x_1,\ldots,x_{d_1+d_2}) :=
	f_1(x_1,\ldots,x_{d_1})f_2(x_{d_1+1},\ldots,x_{d_1+d_2}). \]
\item Given $(f,C)\in\mathscr{P}$, $c\in \N$, and $D\ge \cpx{c}$, we have
$(f\otimes x_1 + c,C+D)\in\mathscr{P}$ where $\otimes$ is as above.
\end{enumerate}

The polynomials obtained this way will be referred to as \emph{low-defect
polynomials}.  If $(f,C)$ is a low-defect pair, $C$ will be called its
\emph{base complexity}.  If $f$ is a low-defect polynomial, we will define its
\emph{absolute base complexity}, denoted $\cpx{f}$, to be the smallest $C$ such
that $(f,C)$ is a low-defect pair.
We will also associate to a low-defect polynomial $f$ the \emph{augmented
low-defect polynomial}
\[ \xpdd{f} = f\otimes x_1; \]
if $f$ is in $d$ variables, this is $fx_{d+1}$.
\end{defn}

In this paper we will only concern ourselves with low-defect pairs
$(f,C)$ where $C=\cpx{f}$, so in the remainder of what follows, we will mostly
dispense with the formalism of low-defect pairs and just discuss low-defect
polynomials.

Note that the degree of a low-defect polynomial is also equal to the number of
variables it uses; see Proposition~\ref{polystruct}.
Also note that augmented low-defect polynomials are never themselves low-defect
polynomials; as we will see in a moment (Proposition~\ref{polystruct}),
low-defect polynomials always have nonzero constant term, whereas augmented
low-defect polynomials always have zero constant term.  We can also observe
that low-defect polynomials are in fact read-once polynomials as discussed in
for instance \cite{ROF}.

Note that we do not really care about what variables a low-defect polynomial is
in -- if we permute the variables of a low-defect polynomial or replace them
with others, we will still regard the result as a low-defect polynomial.  From
this perspective, the meaning of $f\otimes g$ could be simply regarded as
``relabel the variables of $f$ and $g$ so that they do not share any, then
multiply $f$ and $g$''.  Helpfully, the $\otimes$ operator is associative not
only with this more abstract way of thinking about it, but also in the concrete
way it was defined above.

In \cite{paperwo} were proved the following propositions about low-defect
polynomials:

\begin{prop}
\label{polystruct}
Suppose $f$ is a low-defect polynomial of degree $d$.  Then $f$ is a
polynomial in the variables $x_1,\ldots,x_d$, and it is a multilinear
polynomial, i.e., it has degree $1$ in each of its variables.  The coefficients
are non-negative integers.  The constant term is nonzero, and so is the
coefficient of $x_1\cdots x_d$, which we will call the \emph{leading
coefficient} of $f$.
\end{prop}

\begin{proof}
This is Proposition~4.2 from \cite{paperwo}.
\end{proof}

\begin{prop}
\label{basicub}
If $f$ is a low-defect polynomial of degree $d$, then
\[\cpx{f(3^{n_1},\ldots,3^{n_d})}\le \cpx{f}+3(n_1+\ldots+n_d).\]
and
\[\cpx{\xpdd{f}(3^{n_1},\ldots,3^{n_{d+1}})}\le \cpx{f}+3(n_1+\ldots+n_{d+1}).\]
\end{prop}

\begin{proof}
This is a combination of Proposition~4.5 and Corollary~4.12 from \cite{paperwo}.
\end{proof}

The above proposition motivates the following definition:

\begin{defn}
Given a low-defect polynomial $f$ (say of degree $d$) and a number $N$, we will
say that $f$ \emph{efficiently $3$-represents} $N$ if there exist
nonnegative integers $n_1,\ldots,n_d$ such that
\[N=f(3^{n_1},\ldots,3^{n_d})\ \textrm{and}\
\cpx{N}=\cpx{f}+3(n_1+\ldots+n_d).\]
We will say $\xpdd{f}$ efficiently
$3$-represents $N$ if there exist $n_1,\ldots,n_{d+1}$ such that
\[N=\xpdd{f}(3^{n_1},\ldots,3^{n_{d+1}})\ \textrm{and}\ 
\cpx{N}=\cpx{f}+3(n_1+\ldots+n_{d+1}).\]
More generally, we will also say $f$ $3$-represents $N$ if there exist
nonnegative integers $n_1,\ldots,n_d$ such that $N=f(3^{n_1},\ldots,3^{n_d})$.
and similarly with $\xpdd{f}$.
\end{defn}

Note that previous papers \cite{paperalg,paperwo,theory} instead spoke of a
low-defect pair $(f,C)$ efficiently $3$-representing a number $N$; however, as
mentioned in those papers, it is only possible for some $(f,C)$ to efficiently
$3$-represent a number $N$ if in fact $C=\cpx{f}$, so there is no loss here.

In keeping with the name, numbers $3$-represented by low-defect polynomials, or
their augmented versions, have bounded defect.  Let us make some definitions
first:

\begin{defn}
Given a low-defect polynomial $f$ we define $\dft(f)$, the defect of $f$,
to be $\cpx{f}-3\log_3 m$, where $m$ is the leading coefficient of $f$.
\end{defn}

\begin{defn}
Given a low-defect polynomial $f$ of degree $d$, we define
\[\dft_f(n_1,\ldots,n_d) =
\cpx{f}+3(n_1+\ldots+n_d)-3\log_3 f(3^{n_1},\ldots,3^{n_d}).\]
\end{defn}

Then we have:

\begin{prop}
\label{dftbd}
Let $f$ be a low-defect polynomial of degree $d$, and let the numbers
$n_1,\ldots,n_{d+1}$ be nonnegative integers.
\begin{enumerate}
\item We have
\[ \dft(\xpdd{f}(3^{n_1},\ldots,3^{n_{d+1}}))\le \dft_f(n_1,\ldots,n_d),\]
and the difference is an integer.
\item We have \[\dft_f(n_1,\ldots,n_d)\le\dft(f),\]
and if $d\ge 1$, this inequality is strict.
\item The function $\dft_f$ is strictly increasing in each variable, and
\[ \dft(f) = \sup_{n_1,\ldots,n_d} \dft_f(n_1,\ldots,n_d).\]
\end{enumerate}
\end{prop}

\begin{proof}
This is a combination of Proposition~4.9 and Corollary~4.14 from \cite{paperwo}
and Proposition~2.15 from \cite{theory}.
\end{proof}

Importantly, the set of defects coming from a low-defect polynomial of degree
$r$ has order type approximately $\omega^r$; if rather than the actual defects
we use $\dft_f$, then this is exact.  More formally:

\begin{prop}
\label{oldwo}
Let $f$ be a low-defect polynomial of degree $d$.  Then:
\begin{enumerate}
\item The image of $\dft_f$ is a well-ordered subset of $\mathbb{R}$, with
order type $\omega^d$.
\item The set of $\dft(N)$ for all $N$ $3$-represented by the augmented
low-defect polynomial $\xpdd{f}$ is a well-ordered subset of $\mathbb{R}$, with
order type at least $\omega^d$ and at most $\omega^d(\lfloor \delta(f)
\rfloor+1)<\omega^{d+1}$.  The same is true if $f$ is used instead of the
augmented version $\xpdd{f}$.
\end{enumerate}
\end{prop}

\begin{proof}
This is a combination of Propositions 6.2 and 6.3 from \cite{paperwo}.
\end{proof}

The second part of the above proposition follows from the first by means of
theorems about cutting and pasting of well-ordered sets, ultimately due to
Carruth \cite{carruth}.  In particular:

\begin{prop}
\label{cutandpaste}
We have:
\begin{enumerate}
\item If $S$ is a well-ordered set and $S=S_1\cup\ldots\cup S_n$, and $S_1$
through $S_n$ all have order type less than $\omega^k$, then so does $S$.
\item If $S$ is a well-ordered set of order type $\omega^k$ and
$S=S_1\cup\ldots\cup S_n$, then at least one of $S_1$ through $S_n$ also has
order type $\omega^k$.
\end{enumerate}
\end{prop}

\begin{proof}
One may see \cite{carruth} or \cite{wpo} for proofs of these.
\end{proof}

We will need in particular the following variant:

\begin{prop}
\label{interleave}
Suppose $\alpha$ is an ordinal and $S$ is a well-ordered set which can be
written as a finite union $S_1 \cup \ldots \cup S_k$ such that:
\begin{enumerate}
\item The $S_i$ all have order types at most $\omega^\alpha$.
\item If a set $S_i$ has order type $\omega^\alpha$, it is cofinal in $S$.
\end{enumerate}
Then the order type of $S$ is at most $\omega^\alpha$.  In particular, if at
least one of the $S_i$ has order type $\omega^\alpha$, then $S$ has order type
$\omega^\alpha$.
\end{prop}

\begin{proof}
A proof of this can be found in \cite{adcwo} where it is
Proposition 5.4.
\end{proof}

As was noted above, we have $\dft(f(3^{n_1},\ldots,3^{n_d})\le
\dft_f(n_1,\ldots,n_d)$.  Importantly, though, for certain low-defect
polynomials $f$, namely, those with $\dft(f)<\deg f+1$, we can show that
equality holds for ``most'' choices of $(n_1,\ldots,n_d)$ in a certain sense.

Specifically:

\begin{prop}
\label{dump}
Let $f$ be a low-defect polynomial of degree $d$ with $\dft(f)<d+1$.  Define its
``exceptional set'' to be
\[
S:=\{(n_1,\ldots,n_d): \cpx{f(3^{n_1},\ldots,3^{n_d})}_\st<\cpx{f}+3(n_1+\ldots+n_d)\}
\]
Then the set $\{\dft(f(3^{n_1},\ldots,3^{n_d})):(n_1,\ldots,n_d)\in S\}$ has
order type less than $\omega^d$, and therefore so does the set
$\{\dft(\xpdd{f}(3^{n_1},\ldots,3^{n_{d+1}})):(n_1,\ldots,n_d)\in S\}$.
In particular,
for $a\not\equiv \cpx{f}\pmod{3}$, the set
\[ \{\dft(\xpdd{f}(3^{n_1},\ldots,3^{n_{d+1}})):(n_1,\ldots,n_{d+1})\in
\mathbb{Z}^{d+1}_{\ge 0}\} \cap \mathscr{D}^a \] has order type less than
$\omega^d$.
Meanwhile, the set
\[\{\dft(f(3^{n_1},\ldots,3^{n_d})):(n_1,\ldots,n_d)\notin
S\}\] has order type at least $\omega^d$, and thus so does the set
\[ \{\dft(f(3^{n_1},\ldots,3^{n_d})):(n_1,\ldots,n_d)\in \mathbb{Z}^d_{\ge
0}\} \cap \mathscr{D}_\st^{\cpx{f}}; \]
moreover, the supremum of this latter set is equal to $\dft(f)$.
\end{prop}

\begin{proof}
Most of this is direct from Proposition~7.2 from \cite{paperwo}; the only parts
not covered in the statement there there are the statement about
$\{\dft(\xpdd{f}(3^{n_1},\ldots,3^{n_{d+1}})):(n_1,\ldots,n_d)\in S\}$, the
statement regarding $a\not\equiv \cpx{f}\pmod{3}$, and the final statement.

The first of these follows directly from the first part, because
\[\dft(\xpdd{f}(3^{n_1},\ldots,3^{n_{d+1}}))\le\dft(f(3^{n_1},\ldots,3^{n_d}))\]
with the difference being an integer, and that integer can certainly be no more
than $\dft(f(3^{n_1},\ldots,3^{n_d}))\le\dft(f)$.  Thus the set
\[\{\dft(\xpdd{f}(3^{n_1},\ldots,3^{n_{d+1}})):(n_1,\ldots,n_{d+1})\in
\mathbb{Z}^{d+1}_{\ge 0}\} \cap \mathscr{D}^a\] can be covered by finitely many
translates of $\{\dft(f(3^{n_1},\ldots,3^{n_d})):(n_1,\ldots,n_d)\in S\}$ and so
by Proposition~\ref{cutandpaste} has order type less than $\omega^d$.

For the statement about
\[ \{\dft(\xpdd{f}(3^{n_1},\ldots,3^{n_{d+1}})):(n_1,\ldots,n_{d+1})\in
\mathbb{Z}^{d+1}_{\ge 0}\} \cap \mathscr{D}^a \]
with $a\not\equiv C\pmod{3}$, if
$\cpx{\xpdd{f}(3^{n_1},\ldots,3^{n_d})}\equiv a\not\equiv \cpx{f} \pmod{3}$,
then in particular this means that
\[ \cpx{\xpdd{f}(3^{n_1},\ldots,3^{n_{d+1}})} \ne \cpx{f}+3(n_1+\ldots+n_{d+1})
\]
which means that
\[ \cpx{\xpdd{f}(3^{n_1},\ldots,3^{n_{d+1}})} < \cpx{f}+3(n_1+\ldots+n_{d+1}) \]
and therefore that
\[ \cpx{f(3^{n_1},\ldots,3^{n_d})}_\st < \cpx{f}+3(n_1+\ldots+n_d), \]
i.e., that $(n_1,\ldots,n_d)\in S$.  Applying what was proved in the previous
paragraph now proves the statement.

As for the final statement, the set
$\{\dft(f(3^{n_1},\ldots,3^{n_d})):(n_1,\ldots,n_d)\in \mathbb{Z}^d_{\ge
0}\} \cap \mathscr{D}_\st^{\cpx{f}}$ contains $\dft_f(\mathbb{N}^d\setminus S)$
(one may see the proof in \cite{paperwo}) which in turn contains
$\dft_f(\mathbb{N}^d)\setminus\dft_f(S)$.  Since the image of
$\dft_f$ has order type $\omega^d$ while $\dft_f(S)$ has order type less
than $\omega^d$ -- similarly to above, this follows by the initial statement and
Proposition~\ref{cutandpaste} -- it follows that
$\dft_f(\mathbb{N}^d)\setminus\dft_f(S)$ has order type $\omega^d$ and
thus is cofinal in the image of $\dft_f$, and thus has supremum $\dft(f)$;
and the same is true of the larger set
$\{\dft(f(3^{n_1},\ldots,3^{n_d})):(n_1,\ldots,n_d)\in \mathbb{Z}^d_{\ge 0}\}
\cap \mathscr{D}_\st^{\cpx{f}}$ which is also bounded above by $\dft(f)$.
\end{proof}

Finally, one more property of low-defect polynomials we will need is the
following:

\begin{prop}
\label{polycpxbd}
Let $f$ be a low-defect polynomial, and suppose that $a$ is the leading
coefficient of $f$.  Then $\cpx{f}\ge \cpx{a} + \deg f$.  In particular,
$\dft(f) \ge \dft(a) + \deg f$.
\end{prop}

\begin{proof}
This is Proposition~3.24 from \cite{theory}.
\end{proof}

With this, we have the basic properties of low-defect polynomials.

\begin{rem}
Note that one reason nothing is lost here by discarding the formalism of
low-defect pairs is that the low-defect pairs $(f,C)$ we will (implicitly)
concern ourselves with in this paper are ones that satisfy $C-3\log_3 m<\deg
f+1$, where $m$ is the leading coefficient of $f$.  However, by
Proposition~\ref{polycpxbd},
\[ \deg f \le \dft(f) \le C-3\log_3 m < \deg f + 1, \]
thus $C-\cpx{f}=(C-3\log_3 m)-\dft(f)<1$ and so $C=\cpx{f}$.  So if we were to
use low-defect pairs, we would only be using pairs where $C=\cpx{f}$, so we lose
nothing by making this assumption.
\end{rem}

\subsection{Good coverings}
\label{buildsec}

We need one more set of definitions before we can state the theorem that will be
used as the basis of the proof of the main theorem.  We define:

\begin{defn}
A natural number $n$ is called a \emph{leader} if it is the smallest number with
a given defect.  By part (6) of Theorem~\ref{oldprops}, this is equivalent to
saying that either $3\nmid n$, or, if $3\mid n$, then
$\dft(n)<\dft(\frac{n}{3})$, i.e., $\cpx{n}<3+\cpx{\frac{n}{3}}$.
\end{defn}

Let us also define:

\begin{defn}
For any real $s\ge0$, define the set of {\em $s$-defect numbers} $A_s$ to be 
\[A_s := \{n\in\mathbb{N}:\dft(n)<s\}.\]
Define the set of {\em $s$-defect leaders} $B_s$ to be 
\[
B_r:= \{n \in A_s :~~n~~\mbox{is a leader}\}.
\]
\end{defn}

These sets are related by the following proposition from \cite{paperwo}:

\begin{prop}
\label{arbr}
For every $n\in A_s$, there exists a unique $m\in B_s$ and $k\ge 0$ such that
$n=3^k m$ and $\dft(n)=\dft(m)$; then $\cpx{n}=\cpx{m}+3k$.
\end{prop}

Because of this, if we want to describe the set $A_r$, it suffices to describe
the set $B_r$.  Now we can define:

\begin{defn}
\label{goodcover}
For a real number $s\ge0$, a finite set $\sS$ of low-defect polynomials will be
called a \emph{good covering} for $B_s$ if every $n\in B_r$ can be efficiently
$3$-represented by some polynomial in $\sS$ (and hence every $n\in Asr$ can be
efficiently represented by some $\xpdd{f}$ with $f\in \sS$) and if for every
$f\in\sS$, $\dft(f)\le s$, with this being strict if $\deg f=0$.
\end{defn}

This allows us to state the main theorem from \cite{theory}:
\begin{thm}
\label{theory}
For any real number $s\ge 0$, there exists a good covering of $B_s$.
\end{thm}

\begin{proof}
This is Theorem~4.9 from \cite{theory} rewritten in terms of
Definition~\ref{goodcover}, and using low-defect polynomials instead of pairs.
(Any low-defect pairs $(f,C)$ with $C>\cpx{f}$ can be filtered out of a good
covering, since such a pair can never efficiently $3$-represent anything.)
\end{proof}

Note that by Proposition~\ref{polycpxbd}, if $f$ is in a good covering of
$B_s$ with leading coefficient $m$, we must have $\dft(m)+\deg f\le s$.

\section{The integer defect}
\label{secdefn}

In this section we state some basic facts about $D(n)$, what it means, and how
it may be computed.

Let us start by giving another interpretation of what $D(n)$ means:
\begin{prop}
\label{Dinterp}
For a natural number $n$,
\[ D(n) = |\{k: n<E(k)\le E(\cpx{n})\}|. \]
\end{prop}

That is to say, $D(n)$ measures how far down $n$ is among numbers with
complexity $\cpx{n}$, measured by how many values of $E$ one passes as one
counts downwards towards $n$ from the largest number also having complexity
$\cpx{n}$.

\begin{proof}
By definition, $L(n)$ is the largest $k$ such that $E(k)\le n$.  Since $E(k)$ is
strictly increasing, the number of $k$ such that $n<E(k)\le E(\cpx{n})$ is equal
to the difference $\cpx{n}-L(n)$, i.e., $D(n)$.
\end{proof}

So for instance, one has that $D(n)=0$ if and only if $n$ is of the form $E(k)$
for some $k$, i.e., $n$ is the largest number of its complexity; while $D(n)\le
1$ if and only if $n>E(\cpx{n}-1)$, i.e., $n$ is greater than all numbers of
lower complexity.  Numbers $n$ with $D(n)\le 1$ will be discussed more in
Section~\ref{Erk}.

As for properties of the integer defect, it behaves largely analogously to the
real defect:

\begin{prop}
\label{oldpropsD}
We have:
\begin{enumerate}
\item For all $n$, $D(n)\ge 0$.
\item For all $n>1$, $L(3n)=L(n)+3$.
\item For $n>1$ and $k\ge 0$, one has $D(3^k n)\le D(n)$, with equality if and
only if $\cpx{3^k n}=3k+\cpx{n}$.
\item A number $n>1$ is stable if and only if for any $k\ge 0$, $D(3^k n)=D(n)$.
\end{enumerate}
\end{prop}

\begin{proof}
Statement (1) is just the statement that $L(n)\le\cpx{n}$; this follows from the
definition of $L(n)$ as $E(\cpx{n})\ge n$ and so (as $E(k)$ is increasing) one
must have $L(n)\le \cpx{n}$.  And once statement (2) is established, statements
(3) and (4) then follow from that and may be proved in exactly the same way
their analogous statements in Theorem~\ref{oldprops} are proved.  This leaves
just statement (2) to be proved.
Note that, for any $k>1$, $E(k+3)=3E(k)$.  Therefore, for any $k>1$, $E(k+3)\le
3n$ if and only if $E(k)\le n$, and so $L(3n)=L(n)+3$; the only possible
exception to this would be if one had $L(n)=1$, which happens only when $n=1$.
\end{proof}

Note that while the theorem that for any $n$ there is some $k$ such that $3^k n$
is stable was originally proven using the defect $\dft(n)$, it could also just
as well be proven using the integer defect $D(n)$.

We can also of course define a stable variant of $D(n)$:

\begin{defn}
For a positive integer $n$, we define the stable integer defect of $n$, denoted
$D_\st(n)$, to be $D(3^k n)$ for any $k$ such that $3^k n$ is stable.
\end{defn}

Note that Proposition~\ref{oldpropsD} shows that this is well-defined.  We then
have:

\begin{prop}
\label{stoldpropsD}
We have:
\begin{enumerate}
\item $D_\st(n)= \min_{k\ge 0} D(3^k n)$
\item For $n>1$, $D_\st(n)=\cpx{n}_\st-L(n)$
\item $D_\st(n) \le D(n)$, with equality if and only if $n$ is stable or $n=1$
\item For $n>1$, $D(n)-D_\st(n)=\dft(n)-\dft_\st(n)=
\cpx{n}-\cpx{n}_\st$
\end{enumerate}
\end{prop}

\begin{proof}
With the exeption of (4), of which no analogue has previously been mentioned,
these all follow from Proposition~\ref{oldpropsD} and their proofs are exactly
analogous to those of the statements in Proposition~\ref{stoldprops}; meanwhile
(4) follows immediately from (2) and the definition of $D(n)$.
\end{proof}

We then also have the analogue of Proposition~\ref{Dinterp}:

\begin{prop}
For a natural number $n>1$,
\[ D_\st(n) = |\{k: n<E(k)\le E(\cpx{n}_\st)\}|. \]
\end{prop}

\begin{proof}
Once again, by definition, $L(n)$ is the largest $k$ such that $E(k)\le n$.
And since $E(k)$ is strictly increasing, the number of $k$ such that $n<E(k)\le
E(\cpx{n}_\st)$ is equal to the difference $\cpx{n}_\st-L(n)$, which by
Proposition~\ref{stoldpropsD} is $D_\st(n)$.
\end{proof}

\begin{rem}
It may seem strange that $1$ needs to be excluded, given that its special status
goes away when stabilized.  However, $\cpx{1}_\st=0$, and $E(0)$ is not defined,
so $n=1$ must still be excluded from the theorem statement.
\end{rem}

Note, by the way:
\begin{prop}
\label{D=0prop}
For any natural number $n$, $D(n)=0$ if and only if $D_\st(n)=0$.
\end{prop}

\begin{proof}
It's immediate that a number $n$ with $D(n)=0$ is stable and so has $D_\st(n)=0$
(unless $n=1$, in which case one still has $D_\st(n)=0$).  For the reverse, a
number $n$ has $D_\st(n)=0$ if and only if there is some $k$ such that $D(3^k
n)=0$.  However, as the numbers $n$ with $D(n)=0$ are precisely those numbers of
the form $3^k$, $2\cdot3^k$, and $4\cdot 3^k$, we see that if $n$ has
$D_\st(n)=0$, it must itself be of one of these forms, and thus have $D(n)=0$.
\end{proof}

See Corollaries~\ref{D=1cor} and \ref{2stab} for related statements.

Having discussed what $D(n)$ is and how it acts, let's finally discuss how it
may be computed.  The quantity $D(n)$ is just the difference $\cpx{n}-L(n)$.  We
know how to compute $\cpx{n}$, although not necessarily quickly; see \cite{anv}
for the currently best-known algorithm for computing complexity, and
\cite{miller} for the best-known bounds on its runtime.  But the other half,
computing $L(n)$, is very simple and can be done much quicker, because it's
given by the following formula:

\begin{prop}
\label{computL}
For a natural number $n$,
\[ L(n) = \max\{3\lfloor\log_3n\rfloor,\ 
3\left\lfloor\log_3\frac{n}{2}\right\rfloor+2,\ 
3\left\lfloor\log_3\frac{n}{4}\right\rfloor+4,1\}. \]
\end{prop}

\begin{proof}
The quantity $L(n)$ is by definition the largest $k$ such that $E(k)\le n$.  The
largest such $k$ congruent to $0$ modulo $3$ is $3\lfloor\log_3n\rfloor$ (so
long as this quantity is positive; otherwise there is none), the largest such
$k$ congruent to $2$ modulo $3$ is $3\lfloor\log_3\frac{n}{2}\rfloor+2$ (with
the same caveat), the largest such $k>1$ congruent to $1$ modulo $3$ is
$3\lfloor\log_3\frac{n}{4}\rfloor+4$ (again with the same caveat), and of course
the largest such $k$ equal to $1$ is $1$.  So the largest of these is $L(n)$
(and any of them that are not valid positive and thus not a valid $k$ will not
affect the maximum).
\end{proof}

Let us make here a definition that will be useful later:

\begin{defn}
For a natural number $n$, define $R(n)=\frac{n}{E(\cpx{n})}$.  We also define
$R_\st(n)$ to be $R(3^k n)$ for any $k$ such that $3^k n$ is stable, or
equivalently (for $n>1$) as $\frac{n}{E(\cpx{n}_\st)}$.
\end{defn}

This is easily related to the defect, as was done in an earlier paper
\cite{paperwo}:

\begin{prop}
\label{dRformulae}
We have, for $n>1$,
\[
\delta(n)=\left\{ \begin{array}{ll}
-3\log_3 R(n)	& \mathrm{if}\quad \cpx{n}\equiv 0\pmod{3}, \\
-3\log_3 R(n) +2\,\delta(2)
	& \mathrm{if}\quad \cpx{n}\equiv 1\pmod{3}, \\
-3\log_3 R(n) +\delta(2)
	& \mathrm{if}\quad \cpx{n}\equiv 2\pmod{3},
\end{array} \right.
\]
and the same relation (without the $n>1$ restriction) holds between $R_\st(n)$,
$\cpx{n}_\st$, and $\dft_\st(n)$.
\end{prop}

\begin{proof}
The relation between $R(n)$ and $\dft(n)$ is just Proposition A.3 from
\cite{paperwo}, and the proof for the stable case is exactly analogous.
\end{proof}

Now we see that in addition to being easy to compute $L(n)$, it's also simple to
determine $D(n)$ from $\dft(n)$, at least if we know the value of $\cpx{n}$
modulo $3$, which technically is implicit in $\dft(n)$.
First, a definition:

\begin{defn}
Let $a$ be a congruence class modulo $3$ and $k$ be a whole number.  Define
\[ \fak{a}(k) = \left\{
\begin{array}{ll}
k		&\textrm{if}\ k\equiv a\pmod{3}\\
k+\dft(2)	&\textrm{if}\ k\equiv a+1\pmod{3}\\
k+2\dft(2)	&\textrm{if}\ k\equiv a+2\pmod{3}\\
\end{array}
\right.\]
\end{defn}

Now:

\begin{thm}
\label{dtoD}
Let $n>1$ be a natural number.  Then $D(n)$ is equal to the smallest $k$ such
that $\dft(n)\le \fak{\cpx{n}}(k)$.  Moreover, if $n$ is any natural number,
$D_\st(n)$ is equal to the smallest $k$ such that $\dft_\st(n)\le
\fak{\cpx{n}_\st}(k)$.

Since two numbers with the same defect also have the same complexity modulo $3$
(and $\dft(n)=1$ if and only if $n=1$), and the analogous statement is also true
of stable complexity and defect, in particular we have that if
$\dft(n)=\dft(m)$ then $D(n)=D(m)$, and if $\dft_\st(n)=\dft_\st(m)$ then
$D_\st(n)=D_\st(m)$.
\end{thm}

Note in addition that since $\dft(n)=\dft(m)$ implies $\dft_\st(n)=\dft_\st(m)$
(see statement (2) in Proposition~\ref{stoldprops}) one has that if
$\dft(n)=\dft(m)$ then $D_\st(n)=D_\st(m)$.

Theorem~\ref{dtoD} makes precise how $D(n)$ is ``almost $\lceil\dft(n)\rceil$''.
It is, as was noted in the introduction, not the same, but it is the smallest
$k$ such that $\dft(n)\le \fak{\cpx{n}}(k)$, where $\fak{\cpx{n}}(k)$ may not be
exactly $k$ but never differs from it by more than $2\dft(2)<0.215.$

\begin{proof}
We prove only the non-stabilized case as the stabilized case is exactly
analogous.  We assume $n>1$.

From Proposition~\ref{Dinterp}, we can see that $D(n)$ is determined by $R(n)$
and the value of $\cpx{n}$ modulo $3$.  Specifically,
\[ D(n) = \left| \left\{ k : R(n) < \frac{E(k)}{E(\cpx{n})} \le 1 \right\}
\right|,\]
so $D(n)$ is the number of values of $\frac{E(k)}{E(\cpx{n})}$ in $(R(n),1]$.
What are the values of this?  They can be obtained as products of values
$\frac{E(k)}{E(k+1)}$; this is equal to $2/3$ when $k\equiv1\ \textrm{or}\ 2
\pmod{3}$ (for $k>1$) and to $3/4$ when $k\equiv0\pmod{3}$.

Thus, if $\cpx{n}\equiv0\pmod{3}$, $D(n)$ will increase whenever $R(n)$ passes a
value of the sequence $1, \frac{2}{3}, \frac{4}{9}, \frac{1}{3}, \frac{2}{9},
\frac{4}{27}, \frac{1}{9}, \ldots$; if
$\cpx{n}\equiv1\pmod{3}$, whenever it passes a value of the sequence $1,
\frac{3}{4}, \frac{1}{2}, \frac{1}{3}, \frac{1}{4}, \frac{1}{6}, \frac{1}{9}, \ldots$; and if $\cpx{n}\equiv2 \pmod{3}$,
whenever it passes a value of the sequence $1, \frac{2}{3}, \frac{1}{2},
\frac{1}{3}, \frac{2}{9}, \frac{1}{6}, \frac{1}{9}, \ldots$.  (These sequences are just the sequences obtained by taking
products of one of the three shifts of the periodic sequence $\frac{2}{3},
\frac{2}{3}, \frac{3}{4}, \frac{2}{3}, \frac{2}{3}, \frac{3}{4}\ldots$; note
that regardless of which shift is used, the repeating part of the sequence
always has a product of $\frac{1}{3}$, and so the product sequences will always
consist of three interwoven geometric sequences each with ratio $\frac{1}{3}$.)

It just remains, then, to convert these values of $R(n)$ to their equivalents in
defects, which can be done with Proposition~\ref{dRformulae}.  Once this is done
one finds that the values of $\dft(n)$ where $D(n)$ increases are precisely
those listed in the definition of $\fak{\cpx{n}}$, which completes the proof.
\end{proof}

Theorem~\ref{dtoD} will form half the proof of Theorem~\ref{mainthm}, and
its stable analogue, Theorem~\ref{mainthmstab}; it tells us that the values of
$D(n)$ ``switch over'' when $\dft(n)$ is of the form $k$, $k+\dft(2)$, or
$k+2\dft(2)$ depending on the congruence class of $k-\cpx{n}$ modulo $3$.  The
other half the proof is, of course, Theorem~\ref{chgoverpt} (and its stable
analogue, Theorem~\ref{stabchgoverpt}), which will tell us that these changeover
points are exactly the limits of the initial $\omega^k$ defects in
$\mathscr{D}^a$ (or $\mathscr{D}^a_\st$).

\section{The order interpretation of $D(n)$}
\label{mainsec}

In this section we aim to prove Theorem~\ref{chgoverpt} using the methods
described in Section~\ref{intropoly}; combined with Theorem~\ref{dtoD} from the
previous section, this will prove Theorem~\ref{mainthm}.  Really, we want to
prove generalizations:

\begin{thm}
\label{stabchgoverpt}
For any $k\ge0$ and $a$ a congruence class modulo $3$, the order type of
$\mathscr{D}^a\cap[0,\fak{a}(k)]$ and the order type of
$\mathscr{D}_\st^a\cap[0,\fak{a}(k)]$ are both equal to $\omega^k$.
\end{thm}

\begin{thm}
\label{mainthmstab}
Let $n>1$ be a natural number.
Let $\zeta$ be the order type of $\mathscr{D}^{\cpx{n}}\cap[0,\dft(n))$.  Then
$D(n)$ is equal to the smallest $k$ such that $\zeta<\omega^k$.  The same is
true if we replace $\dft(n)$ by $\dft_\st(n)$,
$\mathscr{D}^{\cpx{n}}$ by $\mathscr{D}^{\cpx{n}_\st}_\st$, and $D(n)$ by
$D(n)_\st$.
\end{thm}

Note that the proofs in this section will rely heavily on the results in
Sections~\ref{polysubsec} and \ref{buildsec}.  Before we prove these, though, we
will need a slight elaboration on Proposition~\ref{dump}:

\begin{prop}
\label{dumptype}
Let $f$ be a low-defect polynomial of degree $d$ with $\dft(f)<d+1$.  Then the
order type of the set of all $\dft(N)$ for $n$ $3$-represented by $\xpdd{f}$ is
exactly $\omega^d$.
\end{prop}

\begin{proof}
By Proposition~\ref{dump},
$\{\dft(\xpdd{f}(3^{n_1},\ldots,3^{n_d})):(n_1,\ldots,n_d)\in S\}$ has order
type less than $\omega^d$.  Meanwhile, also by Proposition~\ref{dump}, the set
\[\{\dft(f(3^{n_1},\ldots,3^{n_d})):(n_1,\ldots,n_d)\notin S\}\]
has order type at least $\omega^d$, and is cofinal in $[0,\dft(f))$ (or
$[0,\dft(f)]$ if $\deg f=0$) and therefore in the set of all $\dft(N)$ for $n$
$3$-represented by $\xpdd{f}$.  But in fact, for $(n_1,\ldots,n_d)\notin S$,
one has $\dft(\xpdd{f}(3^{n_1},\ldots,3^{n_{d+1}})=\dft_f(n_1,\ldots,n_d)$,
and so this set (even when $f(3^{n_1},\ldots,3^{n_d})$ is replaced by
$\xpdd{f}(3^{n_1},\ldots,3^{n_{d+1}})$) is a subset of the image of
$\dft_f$, which by Proposition~\ref{oldwo} has order type $\omega^d$.  So
the conditions of Proposition~\ref{interleave} apply, and the union of these two
sets, the set of all $\dft(n)$ for $N$ $3$-represented by $\xpdd{f}$, has order
type at most $\omega^d$.  We already know by Proposition~\ref{oldwo} it has
order type at least $\omega^d$, so this proves the claim.
\end{proof}

We now prove the main theorems of this section.

\begin{proof}[Proof of Theorem~\ref{stabchgoverpt}]
We need to show that the order type of $\mathscr{D}^a\cap[0,\fak{a}(k)]$, as
well as the order type of $\mathscr{D}_\st^a\cap[0,\fak{a}(k)]$, are both equal
to $\omega^k$.  This proof breaks down into two parts, an upper bound and a
lower bound.  Since $\mathscr{D}_\st^a\subseteq\mathscr{D}^a$, it suffices to
prove the upper bound for $\mathscr{D}^a\cap[0,\fak{a}(k)]$, and the lower bound
for $\mathscr{D}_\st^a\cap[0,\fak{a}(k)]$.

We begin with the upper bound.  First, we observe that $\fak{a}(k)$ is not
itself an element of $\mathscr{D}^a$ for any $k>0$.  We can see this as neither
$k+\dft(2)$ nor $k+2\dft(2)$ is a defect for any $k>0$ (such a defect would have
to come from some number $n$ satisfying $3^\ell n=2$ or $3^\ell n=4$ for
$\ell>0$, which is impossible), and similarly no nonzero integer is a defect
except $k=1$, which though an element of $\mathscr{D}$ is by definition excluded
from all three $\mathscr{D}^a$.  Thus
$\mathscr{D}^a\cap[0,\fak{a}(k)]=\mathscr{D}^a\cap[0,\fak{a}(k))$ and we may
concern ourselves with the order type of the latter.

Now we take a good covering $\sS$ of $B_{\fak{a}(k)}$ as per
Theorem~\ref{theory}.  For any $f\in\sS$ with leading coefficient $m$, we have
the inequality $\dft(m)+\deg f\le \dft(f)\le \fak{a}(k)$.  In particular, for
any $f\in \sS$, we have $\deg f\le \lfloor \fak{a}(k)\rfloor = k$.

Suppose now that $\deg f=k$; then there is more we can say.  For in this case,
we have $\dft(m)\le \fak{a}(k)-k\le2\dft(2)$.  Thus
$\dft(m)\in\{0,\dft(2),2\dft(2)\}$ by Proposition~\ref{small3}.
Note that by their respective definitions, $\dft(f)\equiv\dft(m)\pmod{1}$;
and, as noted above, $\dft(f)\ge\deg f=k$, and so
$\dft(f)=k+\dft(m)\in\{k,k+\dft(2),k+2\dft(2)\}$.  Note that
$\dft(f)=k+\dft(m)$ means that
\[ k+\cpx{m}-3\log_3 m = \cpx{f} - 3\log_3 m \]
and therefore $\cpx{f}=k-\cpx{m}$.  Moreover, if $\dft(m)=0$, then $m$ is of the
form $3^\ell$ (for some $\ell>0)$ and $\cpx{m}=3\ell$, if $\dft(m)=\dft(2)$ then
$m$ is of the form $2\cdot3^\ell$ with $\cpx{m}=2+3\ell$, and if
$\dft(m)=2\dft(2)$ then $m$ is of the form $4\cdot3^\ell$ with
$\cpx{m}=4+3\ell$; from this we can conclude that, modulo $3$,
\[
\cpx{f}\equiv\left\{\begin{array}{ll}
k	& \textrm{if}\ \dft(f)=k		\\
k-2	& \textrm{if}\ \dft(f)=k+\dft(2)	\\
k-1	& \textrm{if}\ \dft(f)=k+2\dft(2)	\\
\end{array}\right.\]

Now, let
$T_f=\{\dft(\xpdd{f}(3^{n_1},\ldots,3^{n_{d+1}})):n_1,\ldots,n_{d+1}\ge0\} \cap
\mathscr{D}^a$, where $d=\deg f$.  Then by the assumption that $\sS$ is a good
covering of $B_{\fak{a}(k)}$, we have that
\[ \mathscr{D}^a\cap[0,\fak{a}(k)) = \bigcup_{f\in\sS} T_f. \]
We want to show that the conditions of Proposition~\ref{interleave} hold for the
sets $T_f$, so that we can conclude that $\mathscr{D}^a\cap[0,\fak{a}(k))$ has
order type at most $\omega^k$.  If $\deg f<k$, then, by Proposition~\ref{oldwo},
$T_f$ has order type less than $\omega^k$, and thus so does
$T_f\cap\mathscr{D}^a$.  Meanwhile, if $\deg f=k$, then since $\dft(f)\le
\fak{a}(k)<k+1$, we can apply
Proposition~\ref{dumptype} to conclude that the set of $\dft(N)$ for $N$
$3$-represented by $\xpdd{f}$ has order type
$\omega^k$.  However, if $\dft(f)\ne \fak{a}(k)$, then by the previous paragraph
and Proposition~\ref{dump}, we see that while this has order type $\omega^k$,
$T_f$, which is its intersection with $\mathscr{D}^a$, has order type less than
$\omega^k$.

It remains to check, then, that when $\deg f=k$ and $\dft(f)=\fak{a}(k)$, that
the set $T_f$ is cofinal in $\bigcup_{f\in\sS}
T_f=\mathscr{D}^a\cap[0,\fak{a}(k))$, or in other words, just that it's cofinal
in $[0,\fak{a}(k))$.  But this follows from Proposition~\ref{dump}, which in
fact goes further and states that $T_f\cap\mathscr{D}^a_\st$ is cofinal in
$[0,\dft(f))=[0,\fak{a}(k))$.

Thus, applying Proposition~\ref{interleave}, we conclude that
$\mathscr{D}^a\cap[0,\fak{a}(k))$ has order type at most $\omega^k$.  This
proves the upper bound.

To prove the lower bound, let's consider the low-defect polynomial
\[f = (\cdots((mx_1+1)x_2+1)\cdots)x_k + 1\]
(for a particular $m$ to be chosen shortly)
which has $\cpx{f}=\cpx{m}+k$.  (The upper bound on $\cpx{f}$ is immediate and
the lower bound follows from Proposition~\ref{polycpxbd}.)  For the value of
$m$, we take
\[m=\left\{\begin{array}{ll}
3	&\textrm{if}\ k-a\equiv 0\pmod{3}\\
4	&\textrm{if}\ k-a\equiv 2\pmod{3}\\
2	&\textrm{if}\ k-a\equiv 1\pmod{3},
\end{array}\right.\]
so that $\cpx{m}\equiv a-k\pmod{3}$ and $\cpx{f}\equiv a\pmod{3}$, meaning
$\mathscr{D}_\st^{\cpx{f}}=\mathscr{D}_\st^a$.

Then $\dft(f)=\fak{a}(k)$ and so in particular $\dft(f)<k+1$, meaning once again
we can apply Proposition~\ref{dump} to conclude that the set
\[\{\dft(f(3^{n_1},\ldots,3^{n_k})):(n_1,\ldots,n_k)\in \mathbb{Z}^k_{\ge 0}\}
\cap \mathscr{D}_\st^{\cpx{f}}\] has order type at least $\omega^k$.  Since this
set is bounded above by $\dft(f)=\fak{a}(k)$ and
$\mathscr{D}_\st^{\cpx{f}}=\mathscr{D}_\st^a$, we conclude that the order type
of $\mathscr{D}_\st^a\cap[0,\fak{a}(k))$ is at least $\omega^k$.  This completes
the proof.
\end{proof}

In particular this encompasses Theorem~\ref{chgoverpt}.

\begin{proof}[Proof of Theorem~\ref{chgoverpt}]
This is just a rephrasing of Theorem~\ref{stabchgoverpt} with the application to
$\mathscr{D}^a_\st$ omitted.
\end{proof}

Having proven Theorem~\ref{stabchgoverpt}, we can now combine it with
Theorem~\ref{dtoD} to obtain Theorem~\ref{mainthmstab} and
Theorem~\ref{mainthm}:

\begin{proof}[Proof of Theorem~\ref{mainthmstab}]
By Theorem~\ref{dtoD}, $D(n)$ is equal to the smallest $k$ such that $\dft(n)\le
\fak{\cpx{n}}(k)$.  However, since the order type of
$\mathscr{D}^{\cpx{n}}\cap[0,\fak{\cpx{n}}(k))$ is equal to $\omega^k$, one has
that $\zeta<\omega^k$ if and only if $\dft(n)<\fak{\cpx{n}}(k)$.  Thus $D(n)$ is
equal to the smallest $k$ such that $\zeta<\omega^k$.  The proof for the
stabilized version is similar.
\end{proof}

\begin{proof}[Proof of Theorem~\ref{mainthm}]
This is just the special case of Theorem~\ref{mainthmstab} where we only
consider $\dft(n)$ and not $\dft_\st(n)$.
\end{proof}

\section{Numbers $n$ with $D(n)\le 1$}
\label{Erk}

In the previous section we showed that the numbers with integral defect at most
$k$ correspond to the initial $\omega^k$ defects in each of $\mathscr{D}^0$,
$\mathscr{D}^1$, and $\mathscr{D}^2$.  In this section we take a closer look at
the initial $\omega$, the numbers with integral defect at most $1$, and use this
to generalize Theorem~\ref{rawsthm}.

Let's start by listing all the numbers with integral defect at most $1$:

\begin{thm}
\label{defect1}
A natural number $n$ satisfies $D(n)\le1$ if and only if it can be written in
one of the following forms:
\begin{enumerate}
\item $1$, of complexity $1$
\item $2^a 3^k$ for $a\le10$, of complexity $2a+3k$ (for $a$, $k$ not both zero)
\item $2^a(2^b3^\ell+1)3^k$ for $a+b\le 2$, of complexity $2(a+b)+3(\ell+k)+1$
(for $b$, $\ell$ not both zero).
\end{enumerate}
\end{thm}

\begin{proof}
By Theorem~\ref{dtoD}, any $n$ with $D(n)\le 1$ must have
$\dft(n)\le1+2\dft(2)$.  Theorem~31 from \cite{paper1} gives a classification of
all numbers $n$ with $\dft(n)<12\dft(2)$, together with their complexities;
since $12\dft(2)>1+2\dft(2)$, any $n$ with $D(n)\le 1$ may be found among these.
(One may also use the algorithms from \cite{paperalg} to find such a
classification.)  It is then a straightforward matter to determine which of the
$n$ listed there have $D(n)\le 1$.
\end{proof}

This has an important corollary:

\begin{cor}
\label{D=1cor}
For any natural number $n$, $D(n)=1$ if and only if $D_\st(n)=1$.
\end{cor}

\begin{proof}
From Theorem~\ref{defect1}, we see that if $D(n)\le 1$ then we also have $D(3^k
n)\le 1$, and if $D(3^k n)\le 1$ then we have $D(n)\le 1$; this shows that
$D(n)\le 1$ if and only if $D_\st(n)\le 1$.  Combining this with
Proposition~\ref{D=0prop} proves the claim.
\end{proof}

From this we can conclude:

\begin{cor}
\label{2stab}
For any natural number $n>1$, if $D(n)\le 2$ then $n$ is stable (and so
$D_\st(n)\le 2$).
\end{cor}

\begin{proof}
If $D(n)=0$ or $D(n)=1$, this is Proposition~\ref{D=0prop} or
Corollary~\ref{D=1cor}, respectively.  If $D(n)=2$, then for any $k\ge 0$, if we
had $D(3^k n)<2$, then, by Proposition~\ref{D=0prop} and Corollary~\ref{D=1cor},
we would have $D(n)<2$, contrary to assumption; thus $D(3^k n)=2$ for all $k\ge
0$, i.e., $n$ is stable (by Proposition~\ref{stoldpropsD}).
\end{proof}

Note that the converse, that if $D_\st(n)\le 2$ then $D(n)\le 2$, does not hold; for instance, we can
consider $107$, which has $D_\st(107)=2$ but $D(107)=3$, or $683$, which has
$D_\st(683)=2$ but $D(683)=4$.  (It is easy to verify that these numbers have
stable integer defect at most $2$ because $D(321)=D(2049)=2$; that these numbers
do then have stable integer defect equal to $2$ and not any lower can then be
inferred from Corollary~\ref{2stab}.  Alternately, the stable complexity, and
thus stable integer defect, may be computed with the algorithms from
\cite{paperalg}.)

However, for our purposes, the most important consequence of
Corollary~\ref{D=1cor} is the following rephrasing of it:

\begin{prop}
\label{v3lem}
Let $k>1$ be a natural number and suppose $h$ is a value of $R$ corresponding
to a defect in the initial $\omega$ of $\mathscr{D}^k$.  Then if $hE(k)$ is a
natural number $n$, one has $\cpx{n}=k$, and, moreover, $n>E(k-1)$.
\end{prop}

\begin{proof}
Suppose $hE(k)$ is a natural number $n$.  We must have $n>1$ because having
$h=1/E(k)$ for $k>1$ would by Proposition~\ref{dRformulae} correspond to a
defect which is a nonzero integer, and these (by Proposition~\ref{stoldprops})
do not exist.

Then there is, by defintion of $h$, some number $m>1$ with $\cpx{m}\equiv
k\pmod{3}$ and $R(m)=h$, i.e., $m=hE(\cpx{m})$.  Since $\cpx{m}\equiv k\pmod{3}$
we see that $m=n3^\ell$ for some $\ell\in\mathbb{Z}$, where
$\ell=\frac{\cpx{m}-k}{3}$.  But also we have $D(m)\le 1$.  Therefore, whether
$\ell\ge0$ or $\ell\le0$, we must have $D_\st(n)\le 1$, and so, by
Proposition~\ref{D=0prop} and Corollary~\ref{D=1cor}, we have $D(n)\le 1$.  Then
by Proposition~\ref{oldpropsD}, we have $\cpx{m}=\cpx{n}+3\ell$.  From the
definition of $\ell$ we also have $\cpx{m}=k+3\ell$ and thus we conclude that
$\cpx{n}=k$.  And since $D(n)=1$ this means (by Proposition~\ref{Dinterp}) that
$n>E(k-1)$.
\end{proof}

We can now prove Theorem~\ref{tablethm}:

\begin{proof}[Proof of Theorem~\ref{tablethm}]
Suppose we want to determine the $r$'th largest number of complexity $k$.  This
is equivalent to determining the $r$'th largest value of $R(n)=\frac{n}{E(k)}$
that occurs among numbers $n$ of complexity $k$, which is equivalent to
determining the $r$'th smallest defect $\dft(n)$ that occurs among numbers $n$
of complexity $k$.

Now, we can easily determine the initial values $\alpha_0,\ldots,\alpha_r$ of
$\mathscr{D}^k$; let $h_0,\ldots,h_r$ be the corresponding values of the
function $R$, as given by Proposition~\ref{dRformulae}.  (For instance, for a
way of getting $h_0,\ldots,h_r$ directly rather than going by means of defects,
one may take the numbers $n$ given in Theorem~\ref{defect1}, group them by the
residues of $\cpx{n}$ modulo $3$, then sort them in decreasing order by $R(n)$;
note that the values of $R(n)$ obtained this way for any one congruence class of
$\cpx{n}$ modulo $3$ will have reverse order type $\omega$.) One may see
Tables~\ref{table0r}, \ref{table2r}, and \ref{table1r} for tables of the
resulting values of $h$.  Then certainly, the $r$'th largest number of
complexity $k$ is at most $h_r E(k)$, because the set of values of $R(n)$
occuring for $n$ with $\cpx{n}=k$ is a subset of the values of $R(n)$ occuring
for $n>1$ with $\cpx{n}\equiv k\pmod{3}$.  However, it will only be exactly the
$r$'th largest number of complexity $k$ if all of $h_1$ through $h_r$ do indeed
occur for some $n$ with $\cpx{n}=k$.

But, by Proposition~\ref{v3lem}, this is equivalent to just requiring that all
of the numbers $h_0E(k),\ldots,h_rE(k)$ are indeed whole numbers (and moreover
when this does occur one will have $h_i E(k)>E(k-1)$).  In other words, this is
the same as requiring
\[ k \ge \left\{
\begin{array}{ll}
-3\min_{s\le r} v_3(h_s)     & \textrm{if}\ k\equiv0\pmod{3} \\
-3\min_{s\le r} v_3(h_s) + 4 & \textrm{if}\ k\equiv1\pmod{3} \\
-3\min_{s\le r} v_3(h_s) + 2 & \textrm{if}\ k\equiv2\pmod{3}.
\end{array}
\right.
\]

So we have our $h_{r,a}$, and we can take $K_{r,a}$ to be given by this formula.
(Although since for $K_{0,0}$ it may not may make much sense to take
$K_{0,0}=0$, one may wish to take $K_{0,0}=3$ instead, as we have done in
Table~\ref{table0}.)

Combining this with Tables~\ref{table0r}, \ref{table2r}, and \ref{table1r}
yields Tables~\ref{table0}, \ref{table2}, and \ref{table1}, and proves the
theorem.
\end{proof}

\begin{rem}
While in the proof of Theorem~\ref{tablethm} we have referred to facts proved in
Section~\ref{mainsec}, none of the techniques deployed in that section are
necessary for the proof.  For instance, one can easily verify the values of the
$\overline{\mathscr{D}^a}(\omega)$ by directly determining the initial $\omega$
elements without needing to determine it for all $\omega^k$; indeed
Tables~\ref{table0r}, \ref{table2r}, and \ref{table1r} essentially do this
directly from Theorem~\ref{defect1}.
\end{rem}

\begin{table}[htb]
\caption{Table of $h_r$ for $k\equiv 0\pmod{3}$.}
\label{table0r}
\begin{tabular}{|r|r|r|}
\hline
$r$ & 					$h$ & Corresponding leader \\
\hline
$0$ &					$1$ & $3=3^1 2^0=2^1 3^0+1$ \\
$1$ &					$8/9$ & $8=2^3=2^1(3^1+1)$ \\
$2$ &					$64/81$ & $64=2^6$ \\
$3$ &					$7/9$ & $7=2^13^1+1$ \\
$4$ &					$20/27$ & $20=2^1(3^2+1)$ \\
$5$ &					$19/27$ & $19=2^13^2+1$ \\
$6$ &					$512/729$ & $512=2^9$ \\
$(\mathrm{for}\: n\ge4) \quad 2n-1$ &	$2/3+2/3^n$ & $2^1(3^{n-1}+1)$ \\
$(\mathrm{for}\: n\ge4) \quad 2n$ &	$2/3+1/3^n$ & $2^1 3^{n-1}+1$ \\
\hline
\end{tabular}
\end{table}

\begin{table}[htb]
\caption{Table of $h_r$ for $k\equiv 2\pmod{3}$.}
\label{table2r}
\begin{tabular}{|r|r|r|}
\hline
$r$ &					$h$ & Corresponding leader \\
\hline
$0$ &					$1$ & $2=2^1$ \\
$1$ &					$8/9$ & $16=2^4=2^2(3^1+1)$ \\
$2$ &					$5/6$ & $5=2^23^0+1$ \\
$3$ &					$64/81$ & $128=2^7$ \\
$4$ &					$7/9$ & $14=2^1(2^13^1+1)$ \\
$5$ &					$20/27$ & $40=2^2(3^2+1)$ \\
$6$ &					$13/18$ & $13=2^23^1+1$ \\
$7$ &					$19/27$ & $38=2^1(2^13^2+1)$ \\
$8$ &					$512/729$ & $1024=2^{10}$ \\
$(\mathrm{for}\: n\ge 4)\quad3n-3$ &	$2/3+2/3^n$ & $2^2(3^{n-1}+1)$ \\
$(\mathrm{for}\: n\ge 4)\quad3n-2$ &	$2/3+1/(2\cdot3^{n-1})$ &
						$2^23^{n-1}+1$\\
$(\mathrm{for}\: n\ge 4)\quad3n-1$ &	$2/3+1/3^n$ & $2^1(2^1 3^{n-1}+1)$\\
\hline
\end{tabular}
\end{table}

\begin{table}[htb]
\caption{Table of $h_r$ for $k\equiv 1\pmod{3}$ with $k>1$.}
\label{table1r}
\begin{tabular}{|r|r|r|}
\hline
$r$ &					$h$ & Corresponding leader \\
\hline
$0$ &					$1$ & $4=2^2=3^1+1$ \\
$1$ &					$8/9$ & $32=2^5$ \\
$2$ &					$5/6$ & $10=3^2+1$ \\
$3$ &					$64/81$ & $256=2^8$ \\
$(\mathrm{for}\: n\ge 2)\quad n+2$ &	$3/4+1/(4\cdot3^n)$ & $3^{n+1}+1$ \\
\hline
\end{tabular}
\end{table}

As a final note, it is worth making formal a statement mentioned in
Section~\ref{rawsintro}, that the numbers $hE(k)$ coming from
Theorem~\ref{tablethm} are almost exactly the $n$ with $D(n)\le 1$:

\begin{prop}
\label{coincicor}
A number $n$ has $D(n)\le 1$ if and only if there are some $\ell\ge0$, $k\ge
1$, and $r\ge 0$ such that $k\ge K_{r,k}$ and $3^\ell n = h_{r,k} E(k)$.
\end{prop}

\begin{proof}
We already know that if $k\ge K_{r,k}$ then, if we let $m=h_{r,k}E(k)$, that
$m>E(k-1)=E(\cpx{m}-1)$, i.e., $D(m)\le 1$, and so if $m=3^\ell n$, then
$D(n)\le 1$ by Corollary~\ref{D=1cor}.

Conversely, if $D(n)\le 1$, let $h=R(n)$; then by the construction of the
$h_{r,a}$ in the proof of Theorem~\ref{tablethm}, and the fact that the values
of $R(n)$ for numbers $n$ with $\cpx{n}$ in a fixed congruence class modulo $3$
have reverse order type $\omega$, there is some $r$ such that $h=h_{r,\cpx{n}}$.
We may then take any $k\ge K_{r,\cpx{n}}$ with $k\equiv \cpx{n}\pmod{3}$; then
$3^\ell n = h_{r,\cpx{n}}E(k) = h_{r,k}E(k)$ for $\ell=\frac{k-\cpx{n}}{3}$.
\end{proof}

\subsection*{Acknowledgements}
Work of the author was supported by NSF grants DMS-0943832 and DMS-1101373.


\begin{thebibliography}{99}
\bibitem{adcwo} H.~Altman, Internal Structure of Addition Chains: Well-Ordering,
{\it Theoretical Computer Science} (2017), {\tt doi:10.1016/j.tcs.2017.12.002}
\bibitem{paperalg} H.~Altman, Integer Complexity: Agorithms and Computational
Results, {\tt arXiv:1606.03635}, 2016
\bibitem{paperwo} H.~Altman, Integer Complexity and Well-Ordering, {\it
Michigan Mathematical Journal} {\bf 64} (2015), no.~3, 509--538.
\bibitem{theory} H.~Altman, Integer Complexity: Representing Numbers of Bounded
Defect, {\it Theoretical Computer Science} {\bf 652} (2016), 64--85.
\bibitem{stab} H.~Altman and J.~Arias de Reyna, Integer Complexity,
Stability, and Self-Similarity, in preparation
\bibitem{paper1} H.~Altman and J.~Zelinsky, Numbers with Integer
Complexity Close to the Lower Bound, {\it Integers} {\bf 12} (2012), no.~6,
1093--1125.
\bibitem{anv} J.~Arias de Reyna and J.~Van de Lune, Algorithms for determining
integer complexity, {\tt arXiv:1404.2183}, 2014
\bibitem{Brauer} A.~Brauer, On Addition Chains, {\it Bull. Amer. Math. Soc.}
{\bf 45} (1939), 736--739.
\bibitem{carruth} P.~W.~Carruth, Arithmetic of ordinals with applications to
the theory of ordered abelian groups, {\it Bull.\ Amer.\ Math.\ Soc.} {\bf 48}
(1942), 262--271.
\bibitem{miller} K.~Cordwell, A.~Epstein, A.~Hemmady, S.~J.~Miller,
E.~A.~Palsson, A.~Sharma, S.~Steinerberger, Y.~N.~Truong Vu, On algorithms to
calculate integer complexity, {\tt arXiv:1706.08424}, 2017
\bibitem{wpo} D.~H.~J.~De Jongh and R.~Parikh, Well-partial orderings and
hierarchies, {\it Indag.\ Math.} {\bf 39} (1977), 195--206.
\bibitem{Dellac} H.~Dellac, {\it Interm\'ed. Math.} {\bf 1} (1894), 162--164.
\bibitem{sub1962} A.~A.~Gioia, M.~V.~Subbarao, and M.~Sugunamma, The
Scholz-Brauer Problem in Addition Chains, {\it Duke Math. J.} {\bf 29} (1962),
481--487.
\bibitem{Guy} R.~K.~Guy, 
Some suspiciously simple sequences,
 {\it Amer.\ Math.\ Monthly}, {\bf 93} (1986), 186--190;
 and see {\bf 94} (1987), 965 \& {\bf 96} (1989), 905.
\bibitem{UPINT} R.~K.~Guy, {\it Unsolved Problems in Number Theory},
Third Edition, Springer-Verlag, New York, 2004, pp.~399--400.
\bibitem{TAOCP2} D.~E.~Knuth, {\it The Art of Computer Programming}, Vol.\ 2,
Third Edition, Addison-Wesley, Reading, Massachusetts, pp.~461--485
\bibitem{MP} K.~Mahler and J.~Popken, On a maximum problem in arithmetic
 (Dutch), {\it Nieuw Arch.\ Wiskunde}, (3) {\bf 1} (1953),
 1--15; {\it MR} {\bf 14}, 852e.
\bibitem{Raws} D.~A.~Rawsthorne, How many 1's are needed?, {\it Fibonacci
Quart.} {\bf27} (1989), 14--17; {\it MR} {\bf90b}:11008.
\bibitem{aufgaben} A.~Scholz, Aufgabe 253, Jahresbericht der Deutschen
Mathematikervereinigung, Vol. 47, Teil II, B.~G.~Teubner, Leipzig and Berlin,
1937, pp. 41--42.
\bibitem{subreview} M.~V.~Subbarao, Addition Chains -- Some Results and
Problems, {\it Number Theory and Applications}, Editor R.~A.~Mollin, NATO
Advanced Science Series: Series C, V. 265, Kluwer Academic Publisher Group,
1989, pp.~555--574.
\bibitem{ROF} I.~Volkovich, Characterizing Arithmetic Read-Once Formulae, {\it
ACM Trans. Comput. Theory} {\bf 8} (2015), no.~1, Art.~2, 19~pp.
\bibitem{upbds} J.~Zelinsky, An Upper Bound on Integer Complexity, in
preparation
\end{thebibliography}
\end{document}